\documentclass[11pt,english]{article}

\usepackage[T1]{fontenc}
\usepackage[latin9]{inputenc}
\usepackage{geometry}
\usepackage{babel}
\usepackage{amsmath}
\usepackage{amsthm}
\usepackage{amssymb}
\usepackage{setspace}
\usepackage[unicode=true,pdfusetitle,
 bookmarks=true,bookmarksnumbered=false,bookmarksopen=false,
 breaklinks=false,pdfborder={0 0 0},pdfborderstyle={},backref=false,colorlinks=false]
 {hyperref}

\makeatletter

\providecommand{\tabularnewline}{\\}

\theoremstyle{plain}
\newtheorem{thm}{\protect\theoremname}
  \theoremstyle{plain}
  \newtheorem{cor}[thm]{\protect\corollaryname}
  \theoremstyle{plain}
  \newtheorem{lem}[thm]{\protect\lemmaname}
  \theoremstyle{plain}
  \newtheorem{prop}[thm]{\protect\propositionname}
  \theoremstyle{plain}
  \newtheorem{conjecture}[thm]{\protect\conjecturename}
  \theoremstyle{definition}
  \newtheorem{defn}[thm]{\protect\definitionname}
  \theoremstyle{definition}
  \newtheorem{example}[thm]{\protect\examplename}
  \theoremstyle{remark}
  \newtheorem{rem}[thm]{\protect\remarkname}
  \theoremstyle{remark}
  \newtheorem{claim}[thm]{\protect\claimname}
  \theoremstyle{plain}
  \newtheorem{fact}[thm]{\protect\factname}
  \theoremstyle{remark}
  \newtheorem*{rem*}{\protect\remarkname}

\date{}

\usepackage{appendix}

\usepackage[nameinlink,capitalise,noabbrev]{cleveref}
\AtBeginDocument{\renewcommand{\ref}[1]{\cref{#1}}}
\crefname{appsec}{Appendix}{Appendices}
\creflabelformat{enumi}{#2textup{#1}#3}

\usepackage[parfill]{parskip}
\begingroup
    \makeatletter
    \@for\theoremstyle:=definition,remark,plain\do{%
        \expandafter\g@addto@macro\csname th@\theoremstyle\endcsname{%
            \addtolength\thm@preskip\parskip
            }%
        }
\endgroup

\theoremstyle{plain}
\newtheorem{mythm}{\protect\theoremname}
\renewenvironment{thm}{\begin{mythm}}{\end{mythm}}
\crefname{mythm}{Theorem}{Theorems}
\theoremstyle{definition}
\newtheorem{mydefn}[mythm]{\protect\definitionname}
\renewenvironment{defn}{\begin{mydefn}}{\end{mydefn}}
\theoremstyle{definition}
\newtheorem{myfact}[mythm]{\protect\factname}
\renewenvironment{fact}{\begin{myfact}}{\end{myfact}}
\theoremstyle{definition}
\newtheorem{myexample}[mythm]{\protect\examplename}
\renewenvironment{example}{\begin{myexample}}{\end{myexample}}
\theoremstyle{plain}
\newtheorem{myprop}[mythm]{\protect\propositionname}
\renewenvironment{prop}{\begin{myprop}}{\end{myprop}}
\crefname{myprop}{Proposition}{Propositions}
\theoremstyle{plain}
\newtheorem{mycor}[mythm]{\protect\corollaryname}
\renewenvironment{cor}{\begin{mycor}}{\end{mycor}}
\theoremstyle{plain}
\newtheorem{mylem}[mythm]{\protect\lemmaname}
\renewenvironment{lem}{\begin{mylem}}{\end{mylem}}
\crefname{mylem}{Lemma}{Lemmas}
\theoremstyle{plain}
\newtheorem{myconjecture}{\protect\conjecturename}
\renewenvironment{conjecture}{\begin{myconjecture}}{\end{myconjecture}}
\theoremstyle{definition}
\newtheorem{myrem}[mythm]{\protect\remarkname}
\renewenvironment{rem}{\begin{myrem}}{\end{myrem}}
\theoremstyle{remark}
\newtheorem{myclaim}[mythm]{\protect\claimname}
\renewenvironment{claim}{\begin{myclaim}}{\end{myclaim}}

\crefformat{equation}{#2(#1)#3}

\let\originalleft\left
\let\originalright\right
\renewcommand{\left}{\mathopen{}\mathclose\bgroup\originalleft}
\renewcommand{\right}{\aftergroup\egroup\originalright}
\usepackage{pgfplots}
\usetikzlibrary{pgfplots.groupplots}
\usepackage{verbatim}

\makeatletter
\renewcommand*{\UrlTildeSpecial}{%
  \do\~{%
    \mbox{%
      \fontfamily{ptm}\selectfont
      \textasciitilde
    }%
  }%
}%
\let\Url@force@Tilde\UrlTildeSpecial
\makeatother

\usepackage{tikz}
\tikzstyle{vertex}=[circle,draw=black,fill=black,inner sep=0,minimum size=0.2cm,text=white,font=\footnotesize]
\tikzset{every loop/.style={min distance=50,in=50,out=130,looseness=7}}

\usepackage[labelfont=bf,labelsep=period]{caption}

\usepackage{enumitem}

\makeatother

  \providecommand{\claimname}{Claim}
  \providecommand{\conjecturename}{Conjecture}
  \providecommand{\corollaryname}{Corollary}
  \providecommand{\definitionname}{Definition}
  \providecommand{\examplename}{Example}
  \providecommand{\factname}{Fact}
  \providecommand{\lemmaname}{Lemma}
  \providecommand{\propositionname}{Proposition}
  \providecommand{\remarkname}{Remark}
\providecommand{\theoremname}{Theorem}

\begin{document}

\title{Intercalates and Discrepancy in Random Latin Squares}

\author{Matthew Kwan \thanks{Department of Mathematics, ETH, 8092 Zürich. Email: \href{mailto:matthew.kwan@math.ethz.ch} {\nolinkurl{matthew.kwan@math.ethz.ch}}.}\and
Benny Sudakov\thanks{Department of Mathematics, ETH, 8092 Zürich. Email: \href{mailto:benjamin.sudakov@math.ethz.ch}{\nolinkurl{benjamin.sudakov@math.ethz.ch}}.
Research supported in part by SNSF grant 200021-149111.}}

\maketitle
\global\long\def\RR{\mathbb{R}}

\global\long\def\QQ{\mathbb{Q}}

\global\long\def\HH{\mathbb{H}}

\global\long\def\E{\mathbb{E}}

\global\long\def\Var{\operatorname{Var}}

\global\long\def\CC{\mathbb{C}}

\global\long\def\NN{\mathbb{N}}

\global\long\def\ZZ{\mathbb{Z}}

\global\long\def\GG{\mathbb{G}}

\global\long\def\BB{\mathbb{B}}

\global\long\def\DD{\mathbb{D}}

\global\long\def\cL{\mathcal{L}}

\global\long\def\supp{\operatorname{supp}}

\global\long\def\one{\boldsymbol{1}}

\global\long\def\range#1{\left[#1\right]}

\global\long\def\d{\operatorname{d}}

\global\long\def\falling#1#2{\left(#1\right)_{#2}}

\global\long\def\f{\mathbf{f}}

\global\long\def\im{\operatorname{im}}

\global\long\def\sp{\operatorname{span}}

\global\long\def\rank{\operatorname{rank}}

\global\long\def\sign{\operatorname{sign}}

\global\long\def\mod{\operatorname{mod}}

\global\long\def\id{\operatorname{id}}

\global\long\def\disc{\operatorname{disc}}

\global\long\def\lindisc{\operatorname{lindisc}}

\global\long\def\tr{\operatorname{tr}}

\global\long\def\adj{\operatorname{adj}}

\global\long\def\Unif{\operatorname{Unif}}

\global\long\def\Po{\operatorname{Po}}

\global\long\def\Bin{\operatorname{Bin}}

\global\long\def\Ber{\operatorname{Ber}}

\global\long\def\Geom{\operatorname{Geom}}

\global\long\def\sat{\operatorname{sat}}

\global\long\def\Hom{\operatorname{Hom}}

\global\long\def\vol{\operatorname{vol}}

\global\long\def\floor#1{\left\lfloor #1\right\rfloor }

\global\long\def\ceil#1{\left\lceil #1\right\rceil }

\global\long\def\cond{\,\middle|\,}

\let\polishL\L


\DeclareRobustCommand{\L}{\ifmmode{\mathcal{L}}\else\polishL\fi}

\global\long\def\randL{\boldsymbol{L}}

\global\long\def\randN{\boldsymbol{N}}

\global\long\def\randNt{\boldsymbol{N_{2}}}

\global\long\def\randNk{\boldsymbol{N}_{\!k}}

\begin{abstract}
An \emph{intercalate} in a Latin square is a $2\times2$ Latin subsquare.
Let $\randN$ be the number of intercalates in a uniformly random
$n\times n$ Latin square. We prove that asymptotically almost surely
$\randN\ge\left(1-o\left(1\right)\right)\,n^{2}/4$, and that $\E\randN\le\left(1+o\left(1\right)\right)\,n^{2}/2$
(therefore asymptotically almost surely $\randN\le fn^{2}$ for any
$f\to\infty$). This significantly improves the previous best lower
and upper bounds. We also give an upper tail bound for the number
of intercalates in two fixed rows of a random Latin square. In addition,
we discuss a problem of Linial and Luria on low-discrepancy Latin
squares.
\end{abstract}

\section{Introduction}

An $n\times n$ \emph{Latin square} is an $n\times n$ array of the
numbers between $1$ and $n$ (we call these \emph{symbols}), such
that each row and column contains each symbol exactly once. Latin
squares are a fundamental type of combinatorial design, and have many
essentially equivalent formulations. In their various guises, Latin
squares play an important role in many contexts, ranging from group
theory, to projective geometry, to experimental design, to the theory
of error-correcting codes. An introduction to the vast subject of
Latin squares can be found in \cite{DK15}.

Since Erd\H os and R\'enyi's seminal paper on random graphs \cite{ER59}
and Erd\H os' popularization of the probabilistic method, there has
been great interest in random combinatorial structures of all kinds,
and of course it is natural to consider random Latin squares. In fact
random Latin squares are of more than theoretical interest, due to
the importance of randomization in experimental design (see for example
\cite{DK91}).

The simplest and most natural notion of a random Latin square is the
uniform probability distribution over the set $\L$ of $n\times n$
Latin squares. Random Latin squares are very difficult to study: they
lack independence or any kind of recursive structure, which rules
out many of the techniques used to study binomial random graphs and
random permutations, and there is little freedom to make local changes,
which limits the use of ``switching'' techniques often used in the
study of random regular graphs (see for example \cite{KSVW01}). It
is not even known how to efficiently generate a uniformly random $\randL\in\L$.
Jacobson and Matthews \cite{JM96} and Pittenger \cite{Pitt97} designed
Markov chains on $\L$ which converge to the uniform distribution,
but it is not known if these Markov chains converge rapidly.

Some of the earlier work on random Latin squares concerned algebraic
properties (see for example \cite{Cam92,HJ96}). In this paper we
are more interested in structural questions. An \emph{intercalate}
in a Latin square $L$ is a $2\times2$ Latin subsquare. That is,
it is a pair of rows $i,j$ and a pair of columns $x,y$ such that
$L_{i,x}=L_{j,y}$ and $L_{i,y}=L_{j,x}$. An important statistic
of a Latin square $L$ is the number $N\left(L\right)$ of intercalates
that it contains. Clearly this number is at most $n^{3}/4$, because
each of the $n^{2}$ entries in a Latin square can be involved in
at most $n$ intercalates. Heinrich and Wallis \cite{HW81} proved
that for all $n$ there exist $n\times n$ Latin squares with $\Omega\left(n^{3}\right)$
intercalates (the current best lower bound is $n^{3}/8+O\left(n^{2}\right)$
due to Bartlett \cite{Bar13} and independently Browning, Cameron
and Wanless \cite{BCW14}). In the other direction, in a series of
papers due to Kotzig, Lindner, McLeish, Rosa and Turgeon \cite{KLR75,McL75,KT76},
it was proved that for all orders except $2\times2$ and $4\times4$
there exist Latin squares with no intercalates.

In \cite{MW99} McKay and Wanless conjectured the following.
\begin{conjecture}
\label{conj:mckay-wanless}Let $\randN=N\left(\randL\right)$ be the
number of intercalates in a uniformly random Latin square $\randL\in\L$.
For any fixed $\varepsilon>0$, a.a.s.\footnote{By ``asymptotically almost surely'', or ``a.a.s.'', we mean that
the probability of an event is $1-o\left(1\right)$. Here and for
the rest of the paper, asymptotics are as $n\to\infty$.}
\[
\left(1-\varepsilon\right)\frac{n^{2}}{4}\le\randN\le\left(1+\varepsilon\right)\frac{n^{2}}{4}.
\]
\end{conjecture}
They were able to prove the substantially weaker lower bound that
a.a.s. $\randN\ge n^{3/2-\varepsilon}$ for any $\varepsilon>0$.
Before our paper, the best upper bound was due to Cavenagh, Greenhill
and Wanless \cite{CGW08}, who proved that a.a.s. $\randN\le\left(9/2\right)n^{5/2}$.
The techniques used for these upper and lower bounds are very different,
and we incorporate both to prove the following improved bounds. In
particular we are able to prove the lower bound in \ref{conj:mckay-wanless}.

\begin{thm}
\label{thm:upper+lower}Let $\randN=N\left(\randL\right)$ be the
number of intercalates in a uniformly random Latin square $\randL\in\L$.
First,
\[
\left(1-o\left(1\right)\right)\frac{n^{2}}{4}\le\E\randN\le\left(1+o\left(1\right)\right)\frac{n^{2}}{2}.
\]
Second, for any fixed $\varepsilon>0$ and any function $f\to\infty$,
a.a.s.
\[
\left(1-\varepsilon\right)\frac{n^{2}}{4}\le\randN\le fn^{2}.
\]
\end{thm}
\ref{thm:upper+lower} is an immediate corollary of two theorems that
may be of independent interest, which we discuss in \ref{sec:upper+lower}.

A different property that likely holds a.a.s. for random Latin squares
is that they have ``low discrepancy'' or are ``quasirandom'' in
a certain sense. This is related to a conjecture by Linial and Luria
\cite{LL15}. To state their conjecture, note that $\L$ can more
symmetrically be interpreted as the set of all $n\times n\times n$
zero-one arrays with a single ``1'' in each axis-aligned line. To
be specific, an $n\times n$ Latin square $L$ corresponds to the
$n\times n\times n$ array $A=A\left(L\right)$ where $A_{i,x,q}=1$
if $L_{i,x}=q$. A \emph{box} is a set of the form $T=I\times X\times Q$,
where $I,X,Q\subseteq\range n$. For a box $T$, define its \emph{volume
}$\vol T=\left|I\right|\left|X\right|\left|Q\right|$. Let $N_{T}\left(L\right)$
be the number of ones in $A\left(L\right)$ in the positions in the
box $T$. Linial and Luria's conjecture is as follows.
\begin{conjecture}
\label{conj:linial-luria}There exist arbitrarily large Latin squares
$L$ with the following property. For any box $T=I\times X\times Q$,
\[
\left|N_{T}\left(L\right)-\frac{\vol T}{n}\right|=O\left(\sqrt{\vol T}\right).
\]
\end{conjecture}
That is, Linial and Luria conjecture that there are Latin squares
(zero-one arrays) such that in any box, the density of ones is very
close to the density $1/n$ of ones in the entire $n\times n\times n$
array.

It is natural to expect that in fact the statement of \ref{conj:linial-luria}
holds a.a.s. for a uniformly random Latin square $\randL\in\L$. Linial
and Luria proved the weaker result that a.a.s. every ``empty'' box
$T$ with $N_{T}\left(\randL\right)=0$ has $\vol T\le n^{2}\log^{2}n$.
We are able to give a simple argument showing that random Latin squares
a.a.s. have quite low discrepancy, especially when considering boxes
of volume $\Omega\left(n^{2}\log^{2}n\right)$. This encompasses Linial
and Luria's aforementioned result.
\begin{thm}
\label{thm:weak-discrepancy}For a uniformly random $\randL\in\L$,
we a.a.s. have the following. For any box $T=I\times X\times Q$,
\[
\left|N_{T}\left(\randL\right)-\frac{\vol T}{n}\right|=O\left(\sqrt{\vol T}\log n+n\log^{2}n\right).
\]
\end{thm}
The proof of \ref{thm:weak-discrepancy} is in \ref{sec:discrepancy}.

\section{\label{sec:upper+lower}Outline of the proof of\texorpdfstring{}{ Theorem}
\ref{thm:upper+lower}}

\global\long\def\Lext#1{\mathcal{L}^{*}\left(#1\right)}

\global\long\def\Lk#1{\mathcal{L}_{#1}}

\global\long\def\K{K}

\global\long\def\A{A}

\global\long\def\B{B}

The first new ingredient for the proof of \ref{thm:upper+lower} is
the following upper bound, both in expectation and with high probability,
for the number of intercalates in two rows of a random Latin square.
\begin{thm}
\label{thm:two-rows-intercalate-tail}Let $\randNt=N_{2}\left(\randL\right)$
be the number of intercalates in the first two rows of a uniformly
random Latin square $\randL\in\L$. We have
\[
\E\randNt\le1+o\left(1\right)
\]
and
\[
\Pr\left(\randNt\ge t\right)=e^{-\Omega\left(t\log t\right)}.
\]
\end{thm}
Note that $\E\randN={n \choose 2}\E\randNt$ by linearity of expectation,
so the upper bound on $\E\randN$ in \ref{thm:upper+lower} immediately
follows from \ref{thm:two-rows-intercalate-tail}. (Then, the a.a.s.
upper bound on $\randN$ follows from Markov's inequality). We doubt
that the bound on $\E\randNt$ in \ref{thm:two-rows-intercalate-tail}
is sharp; \ref{conj:mckay-wanless} suggests that $\E\randNt\sim1/2$,
and we expect that moreover $\randNt$ has an asymptotic Poisson distribution
with this mean.

The second ingredient for \ref{thm:upper+lower} is a bound for the
lower tail probability of the number of intercalates in a random Latin
square.
\begin{thm}
\label{thm:lower-tail}There is a constant $C$ such that the following
holds. Let $\randN=N\left(\randL\right)$ be the number of intercalates
in a uniformly random Latin square $\randL\in\L$. Suppose $\varepsilon\ge C\log^{1/3}n/n^{1/6}$.
Then
\[
\Pr\left(\randN<\frac{n^{2}}{4}\left(1-\varepsilon\right)\right)\le\exp\left(-\Omega\left(\varepsilon^{2}\frac{\sqrt{n}}{\log n}\right)\right).
\]
\end{thm}
Clearly \ref{thm:lower-tail} implies the a.a.s. lower bound on $\randN$
in \ref{thm:upper+lower}, and because $\randN\ge0$ this in turn
implies the lower bound on $\E\randN$.

When studying random combinatorial structures with little independence,
an indispensable technique is the analysis of ``switching'' operations
that make local changes to an object. One defines switchings that
affect some parameter in a controllable way, then estimates the number
of ways to switch to and from each object, to understand the relative
likelihood of each possible value of the parameter. Switchings underpin
the proofs of both \ref{thm:two-rows-intercalate-tail} and \ref{thm:lower-tail}.

Latin squares are quite ``rigid'' objects, so one cannot easily
define switching operations that make only small changes to a Latin
square. In \cite{CGW08}, Cavenagh, Greenhill and Wanless managed
to overcome this difficulty when studying two fixed rows of a random
Latin square. They considered switchings that make wide-ranging, complicated
changes to the whole Latin square, but have a controllable effect
on the two rows of interest. We prove \ref{thm:two-rows-intercalate-tail}
with a simpler switching operation in a similar spirit. The details
are in \ref{sec:two-rows}.

To prove \ref{thm:lower-tail}, we use \ref{thm:two-rows-intercalate-tail}
and some ideas from \cite{MW99}. A $k\times n$ \emph{Latin rectangle}
is a $k\times n$ array of the numbers from $1$ to $n$, where each
number appears once in each row and not more than once in each column.
We denote the set of all $k\times n$ Latin rectangles by $\Lk k$.

For $k\le n$, any $k\times n$ Latin rectangle can be extended to
a $n\times n$ Latin square. The number of ways to do this does not
depend too much on the Latin rectangle. Indeed, for a $k\times n$
Latin rectangle $L\in\Lk k$ let $\Lext L\subseteq\L$ be the set
of $n\times n$ Latin squares whose first $k$ rows coincide with
$L$. The following estimate is proved with standard upper and lower
bounds on the permanent. It is essentially the same as \cite[Proposition~4]{MW99}.
\begin{prop}
\label{prop:extend-to-latin-squares}For Latin rectangles $L,L'\in\Lk k$,
\[
\frac{\Lext L}{\Lext{L'}}\le e^{O\left(n\log^{2}n\right)},
\]
uniformly over $k$.
\end{prop}
So, the strategy is to find a lower bound on the number of intercalates
in a random $k\times n$ Latin rectangle (for some $k$ to be determined)
that holds with very high probability. We will then be able to apply
\ref{prop:extend-to-latin-squares} to show that the number of intercalates
in the first $k$ rows of a random Latin square satisfies the same
bound with high probability. We can use the union bound to show that
this holds simultaneously for many choices of $k$ rows, which gives
a lower bound for the total number of intercalates in a random Latin
square.

In \cite{MW99}, McKay and Wanless studied the number of intercalates
in a random Latin rectangle. Using their methods, we will prove the
following estimate.
\begin{lem}
\label{lem:lower-bound-rectangles}There is a constant $C$ such that
the following holds. Let $\randL\in\Lk k$ be a uniformly random $k\times n$
Latin rectangle, conditioned on the event that no row is involved
in more than $\K$ intercalates. Let $\randN=N\left(\randL\right)$
be the number of intercalates in $\randL$. If $k\ge\sqrt{n}$ and
$t\ge Ck^{2}\left(\K/n+k/\K\right)$, then 
\[
\Pr\left(\randN\le\frac{k^{2}}{4}-t\right)\le e^{-\Omega\left(t^{2}/k^{2}\right)}.
\]
\end{lem}
Note that \ref{thm:two-rows-intercalate-tail} and the union bound
imply that with high probability no row is involved in many intercalates,
which will give us an appropriate value of $\K$ with which to apply
\ref{lem:lower-bound-rectangles}. In \ref{sec:rectangles-lower-tail}
we prove \ref{lem:lower-bound-rectangles}, and in \ref{sec:squares-lower-tail}
we give the details of how to combine \ref{prop:extend-to-latin-squares},
\ref{lem:lower-bound-rectangles} and \ref{thm:two-rows-intercalate-tail}
to obtain \ref{thm:lower-tail}.

\section{\label{sec:two-rows}Proof of\texorpdfstring{}{ Theorem} \ref{thm:two-rows-intercalate-tail}}

\global\long\def\s#1{\sigma_{1,2}\left(#1\right)}

\global\long\def\C#1#2{C^{#1}\left(#2\right)}

\global\long\def\c#1#2{c_{#1}\left(#2\right)}

\global\long\def\X#1#2{N^{#1}\left(#2\right)}

\global\long\def\Lk#1{\mathcal{L}\left(#1\right)}

\global\long\def\sw#1#2{\mathrm{turn}_{#1}\left(#2\right)}

\global\long\def\fl#1#2{\mathrm{flip}_{#1}\left(#2\right)}

\global\long\def\jo#1#2{\mathrm{join}_{#1}\left(#2\right)}

\global\long\def\FL#1{\mathrm{FL}\left(#1\right)}

Note that any two rows $i,j$ of a Latin square $L\in\L$ define a
permutation $\sigma_{i,j}\left(L\right)$ on the columns of $L$:
column $x$ maps to the column $y$ with $L_{i,y}=L_{j,x}$. Note
that this permutation is a derangement; it has no fixed points $x\mapsto x$.
We will be concerned with the permutation $\s L$ defined by the first
two rows. This permutation decomposes into cycles, the set of which
we denote $\C{}L$. (For our purposes a cycle is a set of columns).
Let $\C{\alpha}L\subseteq\C{}L$ be the set of cycles of length $\alpha$,
and for a column $x$ let $\c xL\in\C{}L$ be the cycle which contains
$x$. Note that $\C 2L$ is the set of intercalates in the first two
rows of $L$.

We first define two primitive switching operations on a Latin square,
chosen such that they have a controllable effect on the intercalate
count in the first two rows.
\begin{defn}
\label{def:primitive-switchings}Consider a Latin square $L$.

\begin{itemize}
\item For any cycle $c\in\C{}L$, we can obtain a new Latin square $\sw cL$
by exchanging the contents of rows 1 and 2, for each column in $c$.
We also write $\sw xL$ to denote $\sw{c_{x}\left(L\right)}L$.
\item Just as two rows $i,j$ of $L$ define the permutation $\sigma_{i,j}\left(L\right)$,
every two columns $x$ and $y$ also define a permutation $\tau_{x,y}\left(L\right)$
of rows (with row $i$ mapping to the row $j$ which satisfies $L_{j,x}=L_{i,y}$).
In the cycle decomposition of $\tau_{x,y}\left(L\right)$, if rows
$1$ and $2$ are in different cycles $c_{1}$ and $c_{2}$, then
we say $\left\{ x,y\right\} $ is a \emph{flippable} pair, and write
$\left\{ x,y\right\} \in\FL L$. We can obtain a new Latin square
$\fl{\left\{ x,y\right\} }L$ by exchanging column $x$ and $y$ for
each row in $c_{2}$.
\end{itemize}
\end{defn}
\begin{figure}[h]
\begin{centering}
$L=\quad$%
\begin{tabular}{|c|c|c|cc}
\hline 
1 & \textbf{5} & \textbf{3} & \multicolumn{1}{c|}{4} & \multicolumn{1}{c|}{2}\tabularnewline
\hline 
4 & \textbf{3} & \textbf{5} & \multicolumn{1}{c|}{2} & \multicolumn{1}{c|}{1}\tabularnewline
\hline 
2 &  & 4 &  & \tabularnewline
\cline{1-1} \cline{3-3} 
3 &  & 2 &  & \tabularnewline
\cline{1-1} \cline{3-3} 
5 &  & 1 &  & \tabularnewline
\cline{1-1} \cline{3-3} 
\end{tabular}$\qquad\qquad$$\sw{\left\{ 2,3\right\} }L=\quad$%
\begin{tabular}{|c|c|c|cc}
\hline 
1 & \textbf{3} & \textbf{5} & \multicolumn{1}{c|}{4} & \multicolumn{1}{c|}{2}\tabularnewline
\hline 
4 & \textbf{5} & \textbf{3} & \multicolumn{1}{c|}{2} & \multicolumn{1}{c|}{1}\tabularnewline
\hline 
2 &  & 4 &  & \tabularnewline
\cline{1-1} \cline{3-3} 
3 &  & 2 &  & \tabularnewline
\cline{1-1} \cline{3-3} 
5 &  & 1 &  & \tabularnewline
\cline{1-1} \cline{3-3} 
\end{tabular}
\par\end{centering}
\caption{We show the effect of the turn operation. (Not all cells are depicted).
Note that $\left\{ 1,3\right\} $ was not flippable in $L$ but is
flippable in $\protect\sw{\left\{ 2,3\right\} }L$. Also note that
$\protect\s L=\protect\s{\protect\sw{\left\{ 2,3\right\} }L}=\left(1\,4\,5\right)\left(2\,3\right)$.}
\end{figure}
\begin{figure}[h]
\begin{centering}
$L=\quad$%
\begin{tabular}{|c|c|c|cc}
\hline 
1 & 3 & 5 & \multicolumn{1}{c|}{4} & \multicolumn{1}{c|}{2}\tabularnewline
\hline 
\textbf{4} & 5 & \textbf{3} & \multicolumn{1}{c|}{2} & \multicolumn{1}{c|}{1}\tabularnewline
\hline 
\textbf{2} &  & \textbf{4} &  & \tabularnewline
\cline{1-1} \cline{3-3} 
\textbf{3} &  & \textbf{2} &  & \tabularnewline
\cline{1-1} \cline{3-3} 
5 &  & 1 &  & \tabularnewline
\cline{1-1} \cline{3-3} 
\end{tabular}$\qquad\qquad$$\fl{\left\{ 1,3\right\} }L=\quad$%
\begin{tabular}{|c|c|c|cc}
\hline 
1 & 3 & 5 & \multicolumn{1}{c|}{4} & \multicolumn{1}{c|}{2}\tabularnewline
\hline 
\textbf{3} & 5 & \textbf{4} & \multicolumn{1}{c|}{2} & \multicolumn{1}{c|}{1}\tabularnewline
\hline 
\textbf{4} &  & \textbf{2} &  & \tabularnewline
\cline{1-1} \cline{3-3} 
\textbf{2} &  & \textbf{3} &  & \tabularnewline
\cline{1-1} \cline{3-3} 
5 &  & 1 &  & \tabularnewline
\cline{1-1} \cline{3-3} 
\end{tabular}
\par\end{centering}
\caption{We show the effect of the flip operation. Note that $\protect\s L=\left(1\,4\,5\right)\left(2\,3\right)$
and $\protect\s{\protect\fl{\left\{ 1,3\right\} }L}=\left(1\,2\,3\,4\,5\right)$.}
\end{figure}
We will be using the flip operation to merge two cycles into a larger
cycle, and we will be using the turn operation to make a pair of columns
from different cycles flippable, if necessary. To justify this, we
make a number of simple observations about the properties of the turn
and flip operations.
\begin{fact}
\label{fact:observations}The operations in \ref{def:primitive-switchings}
have the following consequences.

\begin{enumerate}
\item Suppose $\left\{ x,y\right\} \in\FL L$, and suppose $\c xL\ne\c yL$,
with say $\c xL\in\C{\alpha}L$ and $\c yL\in\C{\beta}L$. Let $L'=\fl{\left\{ x,y\right\} }L$.
Then 
\[
\c x{L'}=\c y{L'}=\c xL\cup\c yL\in\C{\alpha+\beta}{L'}.
\]
Also, $\C{}L\backslash\left\{ \c xL,\c yL\right\} =\C{}{L'}\backslash\left\{ \c x{L'}\right\} $.
That is, flipping with $x$ and $y$ merges $\c xL$ and $\c yL$
and leaves the other cycles unaffected.
\item If $c\in\C 2L$ is an intercalate, then $\s L=\s{\sw cL}.$ That is,
the turn operation does not change the induced permutation.
\item Suppose $\c xL\ne\c yL$ and $\left\{ x,y\right\} \in\FL L$ (respectively,
$\left\{ x,y\right\} \notin\FL L$). Let $L'=\sw xL$. Then $\left\{ x,y\right\} \notin\FL{L'}$
(respectively $\left\{ x,y\right\} \in\FL{L'}$). That is, the turn
operation changes the flippability of $\left\{ x,y\right\} $.
\item For any cycle $c$, we have $\sw c{\sw cL}=L$. For any $\left\{ x,y\right\} \in\FL L$,
we have\\
$\fl{\left\{ x,y\right\} }{\fl{\left\{ x,y\right\} }L}=L$. That is,
the turn and flip operations are both involutions.
\item Suppose $\left\{ x,y\right\} \in\FL L$ and $\c xL\ne\c yL$, with
$\c yL\in\C 2L$. Let $\sigma'=\s{\fl{\left\{ x,y\right\} }L}$. Then
$\left(\sigma'\right)^{2}\left(x\right)=y$.
\end{enumerate}
\end{fact}
With these observations in mind we can define a compound operation
that merges an intercalate with another cycle, regardless of flippability.
\begin{defn}
For columns $x,y$ with $\c xL\ne\c yL$ and $\c yL\in\C 2L$ define
\[
\jo{x,y}L=\begin{cases}
\fl{\left\{ x,y\right\} }L & \mbox{if }\left\{ x,y\right\} \in\FL L,\\
\fl{\left\{ x,y\right\} }{\sw yL} & \mbox{if }\left\{ x,y\right\} \notin\FL L.
\end{cases}
\]
If also $\c xL\in\C 2L$ this is a \emph{double} join, otherwise it
is a \emph{single }join. Note that a double join is not in general
symmetric in $x$ and $y$; we have $\jo{x,y}L=\jo{y,x}L$ if and
only if $\left\{ x,y\right\} \in\FL L$.
\end{defn}
Let $\X{\alpha}L=\left|\C{\alpha}L\right|$ be the number of cycles
of length $\alpha$. Let $\Lk s\subseteq\L$ be the set of Latin squares
$L$ with $s$ intercalates in the first two rows (that is, with $\X 2L=s$).
We make some observations about the join operation.
\begin{fact}
Single and double joins have the following consequences.

\begin{enumerate}
\item A single join always decreases $\X 2{\cdot}$ by exactly one, and
the merged cycle has length greater than 4.
\item A double join always decreases $\X 2{\cdot}$ by exactly two, and
the merged cycle has length 4.
\item For $L\in\Lk{s+1}$, the number of Latin squares $L'\in\Lk s$ which
we can reach with a single join is 
\[
\left(n-2\X 2L\right)\times2\X 2L=2\left(s+1\right)\left(n-2\left(s+1\right)\right).
\]
(Choose a column $x$ not in an intercalate and a column $y$ in an
intercalate).
\item For $L\in\Lk{s+2}$, the number of Latin squares $L'\in\Lk s$ which
we can reach with a double join is at least 
\[
2\X 2L\times2\left(\X 2L-1\right)/2=2\left(s^{2}+3s+2\right)\ge2s^{2}.
\]
(Choose a column $x$ in an intercalate and a column $y$ in a different
intercalate. Since $\left(x,y\right)$ and $\left(y,x\right)$ may
produce the same join, we then divide by 2 for a lower bound).
\item For $L'\in\Lk s$, the number of Latin squares $L\in\Lk{s+1}$ which
can reach $L'$ with a single join is at most 
\[
2\left(n-2\X 2{L'}-3\X 3{L'}-4\X 4{L'}\right)\le2\left(n-2s\right).
\]
(For $\sigma'=\s{L'}$, choose a column $x$ in a cycle with length
greater than 4, and let $y=\left(\sigma'\right)^{2}\left(x\right)$.
If $\left\{ x,y\right\} \in\FL{L'}$ then flip, and then either turn
or don't).
\item For $L'\in\Lk s$, the number of Latin squares $L\in\Lk{s+2}$ which
can reach $L'$ with a double join is at most 
\[
2\times4\X 4{L'}\le2n.
\]
 (For $\sigma'=\s{L'}$, choose a column $x$ in a 4-cycle, let $y=\left(\sigma'\right)^{2}\left(x\right)$,
flip if possible and then either turn or don't).
\end{enumerate}
\end{fact}
Let $J\left(s\right)$ be the number of ways to single join from a
Latin square in $\Lk{s+1}$ to one in $\Lk s$. That is, $J\left(s\right)$
is the number of pairs $\left(L,L'\right)$ where $L\in\Lk{s+1}$,
$L'\in\Lk s$, and $L$ can be obtained from $L'$ by a single join.
We have 
\begin{align*}
2\left(s+1\right)\left(n-2\left(s+1\right)\right)\left|\Lk{s+1}\right| & =J\left(s\right)\le2\left(n-2s\right)\left|\Lk s\right|,\\
\frac{\left|\Lk{s+1}\right|}{\left|\Lk s\right|} & \le\frac{n-2s}{\left(s+1\right)\left(n-2s-2\right)}.
\end{align*}
Similarly, double-counting the number of ways to double join from
a Latin square in $\Lk{s+2}$ to one in $\Lk s$, we obtain 
\[
\frac{\left|\Lk{s+2}\right|}{\left|\Lk s\right|}\le\frac{n}{s^{2}}.
\]
So, 
\begin{equation}
\frac{\left|\Lk{s+1}\right|}{\left|\Lk s\right|}\le\frac{1}{s+1}\left(1+O\left(\frac{1}{n}\right)\right)\label{eq:join-main-recurrence}
\end{equation}
for $2s\le n/2$, and $\left|\Lk{s+2}\right|/\left|\Lk s\right|\le1$
for $2s\ge n/2$ (for large $n$). It follows that for $t\le n/4$
\begin{align*}
\Pr\left(\randNt=t\right) & \le\frac{\left|\Lk t\right|}{\left|\Lk 0\right|}\le\prod_{s=0}^{t-1}\frac{\left|\Lk{s+1}\right|}{\left|\Lk s\right|}=\frac{1}{t!}e^{O\left(t/n\right)}
\end{align*}
and
\[
\Pr\left(t\le\randNt\le\frac{n}{4}\right)\le O\left(1\right)\sum_{s=t}^{n/4}\frac{1}{s!}\le O\left(\frac{1}{t!}+\frac{1}{\left(t+1\right)!}+\frac{1}{\left(t+2\right)!}\sum_{r=0}^{\infty}\frac{1}{\left(t+2\right)^{r}}\right)=O\left(\frac{1}{t!}\right)=e^{-\Omega\left(t\log t\right)}.
\]
For $t>n/4$ we have $\Pr\left(\randNt=t\right)=O\left(1/\left(\left(n/4\right)!\right)\right)$
and
\begin{align}
\Pr\left(\randNt\ge t\right) & \le O\left(\frac{n}{\left(n/4\right)!}\right)=e^{-\Omega\left(n\log n\right)}=e^{-\Omega\left(t\log t\right)}.\label{eq:many-intercalates-probability-bound}
\end{align}
It therefore follows that $\Pr\left(\randNt\ge t\right)=e^{-\Omega\left(t\log t\right)}$
for all $t$.

We now bound $\E\randNt$. By \ref{eq:join-main-recurrence}, for
$1\le t\le n/4$ we have $t\Pr\left(\randNt=t\right)\le\left(1+o\left(1\right)\right)\Pr\left(\randNt=t-1\right)$,
so using \ref{eq:many-intercalates-probability-bound} and noting
that $\randNt\le n/2$,
\begin{align*}
\E\randNt & \le0\Pr\left(\randNt=0\right)+\left(1+o\left(1\right)\right)\sum_{t=1}^{n/4}\Pr\left(\randNt=t-1\right)+\frac{n}{2}\Pr\left(\randNt>\frac{n}{4}\right)\\
 & \le0+\left(1+o\left(1\right)\right)+\frac{n}{2}e^{-\Omega\left(n\log n\right)}\to1.
\end{align*}

\section{\label{sec:rectangles-lower-tail}Proof of\texorpdfstring{}{ Lemma}
\ref{lem:lower-bound-rectangles}}

\global\long\def\Lk#1{\mathcal{L}_{#1}}

\global\long\def\LkK#1#2{\mathcal{L}_{#1}^{#2}}

\global\long\def\LkKs#1#2#3{\mathcal{L}_{#1}^{#3}\left(#2\right)}

\global\long\def\ro#1#2#3{\mathrm{rot}_{#1}^{#2}\left(#3\right)}

\global\long\def\tw#1#2#3{\mathrm{twist}_{#1}^{#2}\left(#3\right)}

Let $\LkK k{\K}\subseteq\Lk k$ be the set of Latin rectangles $L$
in which no row is involved in more than $\K$ intercalates. (We say
these Latin rectangles are ``good''). Let $\LkKs ks{\K}\subseteq\LkK k{\K}$
be the set of good Latin rectangles with exactly $s$ intercalates.

To prove \ref{lem:lower-bound-rectangles} we use essentially the
same switching as in \cite{MW99}, designed to increase the number
of intercalates by exactly 1.
\begin{defn}
\label{def:twist}Consider a Latin rectangle $L\in\LkK k{\K}$. For
a row $i$ and a cyclically ordered set of columns $\left(x\,y\,z\right)$,
we obtain a new $k\times n$ array $L'=\ro{\left(x\,y\,z\right)}iL$
by swapping the symbols in positions $\left(i,x\right),\,\left(i,y\right),\,\left(i,z\right)$
in a cyclic fashion: $L_{i,x}'=L_{i,z},\,L_{i,y}'=L_{i,x},\,L_{i,z}'=L_{i,y}$.
We call this the \emph{rotate} operation. Note that $L'$ might not
be a Latin rectangle, because we might have caused a column to contain
two of the same symbol. 

Now, we define the \emph{twist }operation. For a Latin rectangle $L\in\LkK k{\K}$,
a row $i$ and distinct columns $x,y,z,x',y',z'$, let $L'=\ro{\left(x\,y\,z\right)}i{\ro{\left(x'\,y'\,z'\right)}iL}$.
Suppose the following conditions are satisfied.

\begin{itemize}
\item The rectangle $L'$ is a Latin rectangle, and it is good (that is,
$L'\in\LkK k{\K}$).
\item The positions $\left(i,y\right),\,\left(i,z\right),\,\left(i,y'\right),\,\left(i,z'\right)$
are involved in no intercalates in $L$ or in $L'$
\item The positions $\left(i,x\right)$ and $\left(i,x'\right)$ are involved
in no intercalates in $L$, and in $L'$ there is an intercalate involving
both $\left(i,x\right)$ and $\left(i,x'\right)$. This is the only
intercalate involving $\left(i,x\right)$ or $\left(i,x'\right)$
in $L'$.
\end{itemize}
\end{defn}
Then we define the twist of $L$ by $\tw{\left\{ \left(x,y,z\right),\left(x',y',z'\right)\right\} }iL=L'$.

\begin{figure}[h]
\begin{centering}
$L=\quad$%
\begin{tabular}{|c|c|c|c|c|c|}
\hline 
\textbf{1} & \textbf{3} & \textbf{5} & \textbf{4} & \textbf{2} & \textbf{6}\tabularnewline
\hline 
\multicolumn{1}{c}{} &  & 2 & 3 & \multicolumn{1}{c}{} & \multicolumn{1}{c}{}\tabularnewline
\cline{3-4} 
\end{tabular}$\qquad\qquad$$\tw{\left\{ \left(3,1,2\right),\left(4,6,5\right)\right\} }1L=\quad$%
\begin{tabular}{|c|c|c|c|c|c|}
\hline 
\textbf{5} & \textbf{1} & \textbf{3} & \textbf{2} & \textbf{6} & \textbf{4}\tabularnewline
\hline 
\multicolumn{1}{c}{} &  & 2 & 3 & \multicolumn{1}{c}{} & \multicolumn{1}{c}{}\tabularnewline
\cline{3-4} 
\end{tabular}
\par\end{centering}
\caption{\label{fig:rotate}We show the effect of the twist operation to create
an intercalate involving $\left(1,3\right)$ and $\left(1,4\right)$.}
\end{figure}

\begin{lem}
\label{lem:forwards-rotate}The number of good Latin rectangles $L'\in\LkKs k{s+1}{\K}$
which we can reach via a twist from a specific good Latin rectangle
$L\in\LkKs ks{\K}$ is at least
\[
\frac{1}{2}k^{2}n^{4}\left(1-O\left(\frac{1}{k}+\frac{k}{n}+\frac{\K}{n}+\frac{s}{k\K}\right)\right).
\]
\end{lem}
\begin{proof}
Let $\Psi\left(L\right)$ be the set of rows of $L$ involved in exactly
$\K$ intercalates. We have $\left|\Psi\left(L\right)\right|\le2s/\K$.
Now, choose rows $i$ and $j$ not in $\Psi\left(L\right)$, in which
we will create an intercalate. There are at least
\[
\left(k-\frac{2s}{\K}\right)\left(k-\frac{2s}{\K}-1\right)=k^{2}\left(1-O\left(\frac{s}{k\K}+\frac{1}{k}\right)\right)
\]
ways to do this. (Since we chose rows involved in at most $\K-1$
intercalates, we do not need to worry about violating the goodness
condition).

Next, choose distinct columns $x,\,y,\,x',\,y'$. To create an intercalate
in columns $x$ and $x'$, let $z'$ be the unique column with $L_{j,x}=L_{i,z'}$,
and let $z$ be the column with $L_{j,x'}=L_{i,z}$. There are $n^{4}\left(1+O\left(1/n\right)\right)$
ways to make these choices, but some of these do not give rise to
a valid twist operation. Let $L'=\ro{\left(x\,y\,z\right)}i{\ro{\left(x'\,y'\,z'\right)}iL}$;
the possible violations are as follows.

\begin{itemize}
\item The symbol $L_{i,x}$ might already appear in column $y$ (so that
$L'$ is not a Latin rectangle). For any $x',y,y'$ there are at most
$k$ choices of $x$ with this property, so we should subtract $kn^{3}$
for our upper bound. Similarly $L_{i,x'}$, $L_{i,y}$, $L_{i,y'}$,
$L_{i,z}$ or $L_{i,z'}$ might appear in column $y'$, $z$, $z'$,
$x$ or $x'$ respectively. We should therefore subtract $6kn^{3}$.
\item We might have $z'\in\left\{ x',y,y'\right\} $. For any $x',y,y'$
there are at most $3$ choices of $x$ that cause this. Similarly
we might have $z\in\left\{ x,y,y'\right\} $. We should subtract $6n^{3}$
to compensate for both.
\item One of the positions $\left(i,x\right)$, $\left(i,x'\right)$, $\left(i,y\right)$,
$\left(i,y'\right)$, $\left(i,z\right)$ or $\left(i,z'\right)$
might already be involved in an intercalate. We should subtract $6\times2\K n^{3}$
to compensate for this.
\item There might be an intercalate involving $\left(i,y\right)$ in $L'$.
Recall that $L_{i,y}'=L_{i,x}$, so this can only occur if for one
of the $k-1$ non-$y$ columns ($w$, say) in $L'$ which contain
the symbol $L_{i,x}$ (in row $q\ne i$, say), we have $L_{i,w}'=L_{q,y}'$.
For any $x,x',y'$ there are at most $k$ choices of $y$ for which
this occurs. Similarly, putting $L_{i,y}$, $L_{i,z}$, $L_{i,x'}$,
$L_{i,y'}$ or $L_{i,z'}$ in position $\left(i,z\right)$, $\left(i,x\right)$,
$\left(i,y'\right)$, $\left(i,z'\right)$ or $\left(i,x'\right)$
respectively might create an intercalate involving that position (other
than the one given by positions $\left(i,x\right)$, $\left(i,x'\right)$,
$\left(j,x\right)$, $\left(j,x'\right)$). Similar logic shows that
for each of the 6 cases, we should subtract $kn^{3}$ to compensate.
\end{itemize}
If the above violations do not occur then\textbf{ }we can use $i,x,y,z,x',y',z'$
to twist, so the number of valid ways to twist is at least
\begin{align*}
 & \frac{1}{2}k^{2}\left(1-O\left(\frac{s}{k\K}+\frac{1}{k}\right)\right)\left(n^{4}\left(1+O\left(\frac{1}{n}\right)\right)-O\left(\K n^{3}\right)-O\left(kn^{3}\right)-O\left(n^{3}\right)\right)\\
 & \quad=\frac{1}{2}k^{2}n^{4}\left(1-O\left(\frac{1}{k}+\frac{k}{n}+\frac{\K}{n}+\frac{s}{k\K}\right)\right).
\end{align*}
(We divide by 2 to compensate for the fact that we can exchange $\left(x\,y\,z\right)$
and $\left(x'\,y'\,z'\right)$ to give the same twist).
\end{proof}
\begin{rem*}
Note that there are a number of simpler switching operations one could
have defined in place of the twist operation. For instance, one could
redefine the rotate operation to use cycles of 2 columns rather than
cycles of 3 columns. However, (the analogue of) \ref{lem:forwards-rotate}
would not hold with this simpler switching operation; there would
be less freedom to choose a way to switch, and in fact there are Latin
rectangles from which it is not possible to create exactly one intercalate
using the simpler switching (see \cite{MW99} for an example). This
situation is analogous to the use of 6-cycle switchings rather than
4-cycle switchings in the analysis of random regular graphs (see for
example \cite{KSVW01}).
\end{rem*}
\begin{lem}
\label{lem:backwards-rotate}The number of good Latin rectangles $L\in\LkKs k{s-1}{\K}$
from which we can twist to a specific good Latin rectangle $L'\in\LkKs ks{\K}$
is at most $2sn^{4}$.
\end{lem}
\begin{proof}
Twisting from $L$ must have created one of the $s$ intercalates
in $L'$ as its main intercalate, operating in one of its two rows.
The columns $\left\{ x,x'\right\} $ are determined by the intercalate
that was created, and there are at most $n^{4}$ choices of $y,y',z,z'$
that could have been used. So the number of Latin rectangles $L$
that can twist to $L'$ is at most $2sn^{4}$.
\end{proof}
We can use \ref{lem:forwards-rotate,lem:backwards-rotate} to give
an upper and lower bound on the number of ways to twist from a Latin
rectangle in $\LkKs k{s-1}{\K}$ to a Latin rectangle in $\LkKs ks{\K}$.
For $s\le k^{2}/4$, $k\ge\sqrt{n}$ and $Ck^{2}\left(\K/n+k/\K\right)\le k^{2}/4$
with large $C$, we obtain 
\[
\frac{\left|\LkKs k{s-1}{\K}\right|}{\left|\LkKs ks{\K}\right|}\le\frac{s}{k^{2}/4}\exp\left(O\left(\frac{\K}{n}+\frac{k}{\K}\right)\right),
\]
so for $0\le s\le k$,
\[
\frac{\left|\LkKs k{k^{2}/4-s}{\K}\right|}{\left|\LkKs k{k^{2}/4}{\K}\right|}\le\prod_{r=0}^{s-1}\frac{\left|\LkKs k{k^{2}/4-r-1}{\K}\right|}{\left|\LkKs k{k^{2}/4-r}{\K}\right|}\le\prod_{r=0}^{s-1}\left(\left(\frac{k^{2}/4-r}{k^{2}/4}\right)\exp\left(O\left(\frac{\K}{n}+\frac{k}{\K}\right)\right)\right).
\]
Now, using the fact that $\left(k^{2}/4-r\right)/\left(k^{2}/4\right)\le\exp$$\left(-r/\left(k^{2}/4\right)\right)$,
we have
\begin{align*}
\Pr\left(\randN=k^{2}/4-s\right) & \le\frac{\left|\LkKs k{k^{2}/4-s}{\K}\right|}{\left|\LkKs k{k^{2}/4}{\K}\right|}\\
 & \le\exp\left(-\left(\sum_{r=0}^{s-1}\Theta\left(\frac{r}{k^{2}}\right)\right)+O\left(s\left(\frac{\K}{n}+\frac{k}{\K}\right)\right)\right)\\
 & =\exp\left(-\Omega\left(\frac{s^{2}}{k^{2}}\right)+O\left(s\left(\frac{\K}{n}+\frac{k}{\K}\right)\right)\right).
\end{align*}
If $t\ge Ck^{2}\left(\K/n+k/\K\right)$ for large $C$, then
\[
\Pr\left(\randN<k^{2}/4-t\right)\le\sum_{s=t}^{k^{2}/4}e^{-\Omega\left(s^{2}/k^{2}\right)}=e^{-\Omega\left(t^{2}/k^{2}\right)}.
\]

\section{\label{sec:squares-lower-tail}Proof of\texorpdfstring{}{ Theorem}
\ref{thm:lower-tail}}

The constant $C$ in the theorem statement will be a function of some
other constant $C_{0}$, to be determined. For some $\varepsilon$
satisfying $C\log^{1/3}n/n^{1/6}\le\varepsilon\le1$, let $k=C_{0}\sqrt{n}\log n/\varepsilon$
and let $\K=\varepsilon n/C_{0}$.

Let $\mathcal{E}$ be the event that none of the first $k$ rows of
$\randL$ are involved in more than $\K$ intercalates, in the Latin
rectangle induced by the first $k$ rows. Certainly $\mathcal{E}$
occurs if every pair of distinct rows (among the first $k$) has at
most $\K/\left(k-1\right)$ intercalates, because for each row there
are $k-1$ possible pairs of rows involving that row. By \ref{thm:two-rows-intercalate-tail}
and the union bound (and symmetry considerations),
\[
1-\Pr\left(\mathcal{E}\right)=k^{2}\exp\left(-\Omega\left(\frac{\K}{\left(k-1\right)}\log\frac{\K}{\left(k-1\right)}\right)\right)=\exp\left(-\Omega\left(\frac{\K}{k}\log\frac{\K}{k}\right)\right)=\exp\left(-\Omega\left(\varepsilon^{2}\frac{\sqrt{n}}{\log n}\right)\right).
\]
Let $\randNk$ be the number of intercalates in the first $k$ rows
of $\randL$, and let $t=\varepsilon k^{2}/8$. Note that $C_{0}\K/n\le\varepsilon$
and $\left(C^{3}/C_{0}^{2}\right)k/\K\le\varepsilon$, so for large
$C_{0}$ and larger $C$ (such that $C_{0}$ and $C^{3}/C_{0}^{2}$
are both much larger than the ``$C$'' in \ref{lem:lower-bound-rectangles}),
the conditions in \ref{lem:lower-bound-rectangles} are satisfied.
Combining \ref{lem:lower-bound-rectangles} and \ref{prop:extend-to-latin-squares},
\[
\Pr\left(\randNk<\left(1-\varepsilon/2\right)k^{2}/4\cond\mathcal{E}\right)=e^{-\Omega\left(\varepsilon^{2}k^{2}\right)}e^{O\left(n\log^{2}n\right)}=e^{-\Omega\left(\varepsilon^{2}k^{2}\right)}
\]
for large $C_{0}$. Note that $\varepsilon^{2}k^{2}\gg\varepsilon^{2}\sqrt{n}/\log n$
so 
\begin{equation}
\Pr\left(\randNk<\left(1-\varepsilon/2\right)k^{2}/4\right)\le\exp\left(-\Omega\left(\varepsilon^{2}\frac{\sqrt{n}}{\log n}\right)\right)\label{eq:subrectangle-lower-tail}
\end{equation}
unconditionally. To transfer this result from the first $k$ rows
to the whole of $\randL$, we need the following covering lemma.

\global\long\def\randF#1{\boldsymbol{F}_{\!#1}}

\begin{lem}
\label{lem:covering-rows}For any $k\ll n$ and any $M\gg\left(n\log n/k\right)^{2}$
there exist $k$-subsets $F_{1},\dots,F_{M}$ of $\range n$, such
that every pair $\left\{ i,j\right\} \subseteq\range n$ is included
in 
\[
M\left(\frac{k}{n}\right)^{2}\left(1+O\left(\log n/\left(\sqrt{M}k/n\right)+k/n\right)\right)=M\left(\frac{k}{n}\right)^{2}\left(1+o\left(1\right)\right)
\]
 of the $F_{i}$.
\end{lem}
\begin{proof}
Let $\randF 1,\dots,\randF M$ be independent uniformly random sets
of $k$ rows. For a given pair of rows and some index $i$, the probability
that $\randF i$ contains that pair is 
\[
p={n \choose k-2}/{n \choose k}=\left(\frac{k}{n}\right)^{2}\left(1+O\left(\frac{k}{n}\right)\right).
\]
By the Chernoff bound and the union bound, a.a.s. every pair is contained
in 
\[
Mp+O\left(\sqrt{Mp}\log n\right)=Mp\left(1+O\left(\frac{\log n}{\sqrt{Mp}}\right)\right).
\]
of the $\randF i$. Therefore there exists a specific choice of the
$F_{i}$s that satisfies the requirements of the lemma.
\end{proof}
We apply \ref{lem:covering-rows} with $M=n^{2}$, say, to obtain
sets $F_{1},\dots,F_{M}$. By the union bound and symmetry, the subrectangle
given by the rows of each $F_{i}$ contains at least $\left(1-\varepsilon/2\right)k^{2}/4$
intercalates, except with the probability in \ref{eq:subrectangle-lower-tail}.
Noting that $k/n+\log n/\left(\sqrt{M}k/n\right)\ll\varepsilon$,
this implies 
\begin{align*}
\randN & \ge\frac{M\left(1-\varepsilon/2\right)k^{2}/4}{M\left(k/n\right)^{2}\left(1+o\left(\varepsilon\right)\right)}\\
 & \ge\left(1-\varepsilon\right)\frac{n^{2}}{4}.
\end{align*}

\section{\label{sec:discrepancy}Proof of\texorpdfstring{}{ Theorem} \ref{thm:weak-discrepancy}}

\global\long\def\Gd#1{\mathcal{G}_{#1}}

\global\long\def\Gm#1{\mathcal{G}_{#1}^{*}}

\global\long\def\randG{\boldsymbol{G}}

\global\long\def\randB{\boldsymbol{B}}

\global\long\def\GGB#1#2{\BB\left(#1,#2\right)}

Fix a box $T=I\times X\times Q$ (there are $\left(2^{n}\right)^{3}=8^{n}$
possible choices). We will show that the bound on $N_{T}\left(\randL\right)$
in \ref{thm:weak-discrepancy} fails with probability $o\left(8^{-n}\right)$,
which will allow us to apply the union bound over choices of $T$.

For a Latin square $L$, we define a bipartite graph $G_{Q}\left(L\right)$
as follows. Both parts have $n$ vertices (we abuse notation and say
the vertex set is $\range n\sqcup\range n$); one of the parts is
identified with the set of rows of the Latin square and the other
part is identified with the set of columns. For each row $i$ and
column $x$ such that $L_{i,x}\in Q$, we put an edge between $i$
and $x$ in $G_{Q}\left(L\right)$. Now, the number of ones $N_{T}\left(L\right)$
in $T$ is just the number of edges $e_{G_{Q}\left(L\right)}\left(I,X\right)$
between $I$ and $X$ in $G_{Q}\left(L\right)$.

Let $\Gd d$ be the set of $d$-regular bipartite graphs on $\range n\sqcup\range n$.
For $G\in\Gd{\left|Q\right|}$, let $\left|\L^{*}\left(G\right)\right|$
be the number of Latin squares $L$ with $G_{Q}\left(L\right)=G$.
In a similar way to \ref{prop:extend-to-latin-squares}, we can use
standard bounds on the permanent to prove that $\left|\L^{*}\left(G\right)\right|$
does not vary very much with $G$.
\begin{prop}
\label{prop:extend-graphs-to-latin-squares}For a set of symbols $Q$
and $\left|Q\right|$-regular bipartite graphs $G$ and $G'$,
\[
\frac{\left|\L^{*}\left(G\right)\right|}{\left|\L^{*}\left(G'\right)\right|}\le e^{O\left(n\log^{2}n\right)}
\]
uniformly over $Q$.
\end{prop}
For completeness, we provide a proof of \ref{prop:extend-graphs-to-latin-squares}.
\begin{proof}
Note that we can interpret a Latin square as a 1-factorization of
$K_{n,n}$ (that is, a proper edge colouring with $n$ labelled colour
classes). The correspondence is that the edge between vertex $i$
in the first part and vertex $x$ in the second part receives colour
$q$ if $L_{i,x}=q$. From this point of view, a Latin square in $\Lext G$
is uniquely defined by a 1-factorization of $G$, and a 1-factorization
of the complement of $G$. Let $\Phi\left(G\right)$ be the number
of 1-factorizations of $G$; it suffices to prove that $\Phi\left(G\right)/\Phi\left(G'\right)\le e^{O\left(n\log^{2}n\right)}$.

Let $\phi\left(G\right)$ be the number of 1-factors (perfect matchings)
of a graph $G$. The Egorychev-Falikman theorem \cite{Ego81,Fal81}
(previously known as the Van der Waerden conjecture) and Br\'egman's
theorem \cite{Bre73} (previously known as Minc's conjecture) give
lower and upper bounds on $\phi\left(G\right)$ for a $d$-regular
bipartite graph $G$:
\[
n!\left(\frac{d}{n}\right)^{n}\le\phi\left(G\right)\le\left(d!\right)^{n/d}.
\]

We can therefore give bounds on the number of ways to choose a 1-factorization
by choosing its 1-factors one-by-one:
\[
\prod_{k=1}^{d}n!\left(\frac{k}{n}\right)^{n}\le\Phi\left(G\right)\le\prod_{k=1}^{d}\left(k!\right)^{n/k}.
\]
Now, Stirling's inequality gives
\[
\frac{\left(k!\right)^{n/k}}{n!\left(k/n\right)^{n}}\le\frac{\left(\Theta\left(\sqrt{k}\left(k/e\right)^{k}\right)\right)^{n/k}}{\sqrt{n}\left(n/e\right)^{n}\left(k/n\right)^{n}}\le\exp\left(O\left(\frac{n\left(\log k+1\right)}{k}\right)\right),
\]
so using the approximation $\sum_{i=1}^{d}1/i=\Theta\left(\log d+1\right)$
for the harmonic series,
\[
\frac{\Phi\left(G\right)}{\Phi\left(G'\right)}\le\prod_{k=1}^{d}\frac{\left(k!\right)^{n/k}}{n!\left(k/n\right)^{n}}=e^{O\left(n\left(\log^{2}d+1\right)\right)}=e^{O\left(n\log^{2}n\right)}
\]
as desired.
\end{proof}
The upshot of \ref{prop:extend-graphs-to-latin-squares} is that $G_{Q}\left(\randL\right)$
is not too far from the uniform distribution on $\Gd{\left|Q\right|}$,
and events that hold with very high probability for a uniformly random
$\randG\in\Gd{\left|Q\right|}$ also hold with very high probability
for $G_{Q}\left(\randL\right)$.

It is possible to obtain discrepancy tail bounds for random regular
(bipartite) graphs using switchings of the type in \cite[Theorem~2.2]{KSVW01}.
Such a bound would nearly provide the result we are after (although
there would be difficulties for very dense graphs). However, at the
range of probabilities we are interested in, regular bipartite graphs
comprise a non-negligible proportion of all bipartite graphs with
the appropriate number of edges, and (modulo an enumeration theorem
for regular bipartite graphs) this enables a simpler approach. Let
$\GGB np$ be the random graph distribution on the vertex set $\range n\sqcup\range n$,
where each of the $n^{2}$ possible edges between the parts are present
with independent probability $p$.
\begin{lem}
\label{lem:Gnp-regular}For any $d$ (potentially depending on $n$),
let $p=d/n$. The probability a random graph $\randB\in\GGB np$ is
$d$-regular is $e^{-O\left(n\log n\right)}$. Also, conditioning
on this event gives the uniform distribution on $\Gd d$.
\end{lem}
To prove \ref{lem:Gnp-regular} we will use the following estimate
due to Ordentlich and Roth \cite[Proposition~2.2]{OR00}.
\begin{thm}
\label{thm:regular-enumeration}For any $d$ (potentially depending
on $n$), let $p=d/n$. The number $\left|\Gd d\right|$ of $d$-regular
bipartite graphs on $\range n\sqcup\range n$ is at least
\[
{n \choose d}^{2n}\left(p^{p}\left(1-p\right)^{1-p}\right)^{n^{2}}.
\]
\end{thm}
We remark that a precise asymptotic estimate for $\left|\Gd d\right|$
will very soon become available, due to some soon-to-be-published
developments by Liebenau and Wormald \cite{LW16} and independently
Isaev and McKay \cite{IM17}.
\begin{proof}[Proof of \ref{lem:Gnp-regular}]
The probability $\randB$ has exactly $dn=pn^{2}$ edges is 
\[
{n^{2} \choose pn^{2}}p^{pn^{2}}\left(1-p\right)^{\left(1-p\right)n^{2}}\asymp\frac{1}{n\sqrt{p\left(1-p\right)}}=e^{-o\left(n\right)}.
\]
(here we used Stirling's approximation). By symmetry, each graph with
$dn$ edges is equally likely. By \ref{thm:regular-enumeration},
the fraction of such graphs which are $d$-regular is
\[
{n \choose pn}^{2n}\left(p^{p}\left(1-p\right)^{1-p}\right)^{n^{2}}\left/{n^{2} \choose pn^{2}}\right.=\left(O\left(p\left(1-p\right)n\right)\right)^{-n}\ge e^{-O\left(n\log n\right)}.\tag*{\qedhere}
\]
\end{proof}
Now, discrepancy in $\GGB np$ (for $p=\left|Q\right|/n$) is very
easy to study. Indeed, for $\randB\in\GGB np$ the law of $e_{\randB}\left(I,X\right)$
is the binomial distribution $\Bin\left(\left|I\right|\left|X\right|,p\right)$
with mean $\left|I\right|\left|X\right|p=\vol T/n$. Let $\randG\in\Gd{\left|Q\right|}$
be a uniformly random $\left|Q\right|$-regular bipartite graph. By
a binomial large deviation inequality (for example \cite[Theorem~2.1]{JLR00}),
\ref{prop:extend-graphs-to-latin-squares} and \ref{lem:Gnp-regular},
we have
\begin{align*}
\Pr\left(\left|N_{T}\left(\randL\right)-\frac{\vol T}{n}\right|>t\right) & =\Pr\left(\left|e_{G_{Q}\left(\randL\right)}\left(I,X\right)-\frac{\vol T}{n}\right|>t\right)\\
 & \le\Pr\left(\left|e_{\randG}\left(I,X\right)-\frac{\vol T}{n}\right|>t\right)e^{O\left(n\log^{2}n\right)}\\
 & \le\Pr\left(\left|e_{\randB}\left(I,X\right)-\frac{\vol T}{n}\right|>t\right)e^{O\left(n\log^{2}n+n\log n\right)}\\
 & =\exp\left(-\Omega\left(\frac{t^{2}}{\vol T/n+t}\right)+O\left(n\log^{2}n\right)\right).
\end{align*}
If $t$ is a large multiple of $\sqrt{\vol T}\log n+n\log^{2}n$,
then this probability is $e^{-\Omega\left(n\log^{2}n\right)}=o\left(8^{-n}\right)$.

\section{Concluding remarks}

We have shown that the number of intercalates $\randN$ in a uniformly
random $n\times n$ Latin square a.a.s. satisfies $\left(1-o\left(1\right)\right)n^{2}/4\le\randN\le fn^{2}$,
for any $f\to\infty$, and we showed that $\left(1+o\left(1\right)\right)n^{2}/4\le\E\randN\le\left(1+o\left(1\right)\right)n^{2}/2$.
In doing so we obtained an exponentially-decaying estimate for the
lower tail of $\randN$ and an exponential upper-tail estimate for
the number of intercalates in two fixed rows. We also proved that
random Latin squares typically have relatively low discrepancy.

There are a number of related problems that remain open. First, there
is the task of reducing the a.a.s. upper bound on $\randN$ to $\left(1+o\left(1\right)\right)n^{2}/4$
or at least to $O\left(n^{2}\right)$. The most obvious way of approaching
this would be to imitate our proof of the lower bound, and show that
for some $k$ satisfying $\sqrt{n}\log n\ll k$, with very high probability
a random $k\times n$ Latin rectangle does not have too many intercalates.
The tools from \cite{MW99} can accomplish this conditioned on the
nonexistence of certain ``problematic configurations'' of intercalates,
but showing these configurations are unlikely appears to be a surprisingly
difficult task.

Second, there is the problem of understanding the existence and number
of substructures other than intercalates in random Latin squares.
McKay and Wanless \cite{MW99} conjecture that the number of $3\times3$
Latin subsquares should have expectation $\Theta\left(1\right)$,
and similar logic would suggest that a.a.s. there are no Latin subsquares
of larger order. A proof of either of these facts would be interesting.

Third, there is the task of making further progress towards \ref{conj:linial-luria}.
Even a slight improvement over our \ref{thm:weak-discrepancy} would
be interesting, because such an improvement would have to avoid the
error introduced by the permanent estimates in \ref{prop:extend-to-latin-squares,prop:extend-graphs-to-latin-squares}.

Finally, it would be interesting to prove analogous results for more
general types of random designs, such as Latin cubes or Steiner triple
systems. See for example the recent work of Kwan \cite{Kwa16} on
random Steiner triple systems.

\begin{appendices}
\crefalias{section}{appsec}

\end{appendices}
\end{document}